\documentclass[journal,twoside,web]{ieeecolor}
\usepackage{generic}
\let\labelindent\relax
\usepackage{enumitem}
\def\BibTeX{{\rm B\kern-.05em{\sc i\kern-.025em b}\kern-.08em
    T\kern-.1667em\lower.7ex\hbox{E}\kern-.125emX}}
\usepackage{graphics} % for pdf, bitmapped graphics files
\usepackage{epsfig} % for postscript graphics files
\usepackage{mathptmx} % assumes new font selection scheme installed
\usepackage{times} % assumes new font selection scheme installed
\usepackage{amsmath} % assumes amsmath package installed
\usepackage{amssymb}  % assumes amsmath package installed

\newcommand{\R}{\mathbb{R}}
\usepackage{mathtools}

\usepackage{xcolor}

\usepackage{comment}
\usepackage{algorithmic}
\usepackage{algorithm}

\newtheorem{definition}{Definition}
\newtheorem{assumption}{Assumption}
\newtheorem{lemma}{Lemma}

\newtheorem{theorem}{Theorem}
\newtheorem{corollary}{Corollary}
\newtheorem{remark}{Remark}

\newcommand{\fin}{\triangleleft}

\allowdisplaybreaks

\begin{document}
%%%%%%%%%%%%%%%%%
% Edited by Miyano
\title{Contraction Analysis of Continuation Method for Suboptimal Model Predictive Control}
\author{Ryotaro Shima, Yuji Ito, \IEEEmembership{Senior Member, IEEE}, and Tatsuya Miyano 
\thanks{Ryotaro Shima, Yuji Ito, and Tatsuya Miyano are with 
Toyota Central R\&D Labs., Inc., 41-1 Yokomichi, Nagakute, Aichi 480-1192, Japan 
(e-mail: ryotaro.shima@mosk.tytlabs.co.jp; ito-yuji@mosk.tytlabs.co.jp; miyano@mosk.tytlabs.co.jp).}}
%%%%%%%%%%%%%%%%%

\bstctlcite{IEEEexample:BSTcontrol}  % added by Shima

\maketitle
\thispagestyle{empty}
\pagestyle{empty}

%%%%%%%%%%%%%%%%%%%%%%%%%%%%%%%%%%%%%%%%%%%%%%%%%%%%%%%%%%%%%%%%%%%%%%%%%%%%%%%%
\begin{abstract}

This letter analyzes the contraction property of the nonlinear systems controlled by suboptimal model predictive control (MPC) using the continuation method.
We propose a contraction metric that reflects the hierarchical dynamics inherent in the continuation method. 
We derive a pair of matrix inequalities that elucidate the impact of suboptimality on the contraction of the optimally controlled closed-loop system.
A numerical example is presented to verify our contraction analysis. 
Our results are applicable to other MPCs than stabilization, including economic MPC.

\end{abstract}

%%%%%%%%%%%%%%%%%
% Added by Miyano
%%%%%%%%%%%%%%%%%

%%%%%%%%%%%%%%%%%%%%%%%%%%%%%%%%%%%%%%%%%%%%%%%%%%%%%%%%%%%%%%%%%%%%%%%%%%%%%%%%
\section{INTRODUCTION}

%%%%%%%%%%%%%%%%%
% Edited by Miyano
\IEEEPARstart{M}{odel} 
%%%%%%%%%%%%%%%%%
predictive control (MPC) has attracted significant attention in the industry owing to its capability of handling nonlinearity of the plant model and constraints of input and state \cite{mpc, automotive_mpc}.
In nonlinear MPC, limited computational resources cause lack of optimality of input series \cite{liao-mcpherson, diehl_inexact, realtime}.
In this letter, we refer to such MPC as suboptimal MPC \cite{feasibility, suboptimal_empc}.
In the past two decades, the closed-loop stability of suboptimal MPC has been investigated \cite{feasibility, gap, jones, lygeros}.
However, these methods are only applicable to stabilization.

Continuation method \cite{cgmres, continuation} is an effective algorithm for suboptimal MPC.
It compensates the time evolution of the optimal inputs by solving a linear equation system, and reaches optimality: see also Remark \ref{remark:feedforward}.
However, the extent to which the suboptimality affects the closed-loop behavior remains unknown.

Contraction theory is a framework for the stability of dynamical systems \cite{bookbullo, hiro} and provides a matrix inequality that specifies the convergence of difference between trajectories \cite{manchester2017}.
Although it has been applied to MPC \cite{kohler, wang2017}, contraction analysis of suboptimal MPC is still open.

This letter discusses the convergence of the suboptimally controlled trajectory to optimally controlled trajectory by means of contraction theory.
The difficulty is that the contraction analysis becomes intractable by randomly choosing the contraction metric.
To avoid this, we propose a contraction metric that reflects a certain hierarchical dynamics inherent in the continuation method.
Using the proposed metric, we present a pair of matrix inequalities that express the impact of the suboptimality on the contraction of the optimally controlled dynamics.
We also verify our analysis numerically.
Our results are applicable to other MPCs than stabilization, including economic MPC.

The remainder of this letter is organized as follows:
Section II presents preliminaries of the contraction theory and the continuation method.
Section III outlines our main results, including our proposed contraction metric and the matrix inequalities.
A numerical verification of our analysis is presented in Section IV.

\subsection{Related Work}

Studies \cite{feasibility, diehl_inexact, liao-mcpherson, jones, lygeros} shows a general stability of the closed-loop system.
However, its problem setting is limited to the stabilization to the origin.

Another study \cite{dorfler} investigates the stability of hierarchical dynamics.
However, it does not include any mention of MPC, nor does it include cases in which the plant or optimal control problem is time-varying, which is often considered in MPC.

Theorem 3 in \cite{tvo_contraction} discusses contraction of optimization dynamics that has compensation of time derivative of optimal solution,
though it is limited to optimization and does not discuss the case of connection to dynamical plant.

To the best of our knowledge, no other literature has discussed the contraction analysis of the continuation method.

\subsection{Notation}

$\langle S \rangle \coloneqq S + S^\top$ for a square matrix $S$.
We denote the interval between $a$ and $b$ on $\R$ by $[a, b]$, the $i$-th element of vector (or vector-valued function) $v$ by $v_i$, and an identity matrix of dimension $n$ by $I_n$.
Given $n$ scalars $a^1, \ldots, a^n$, we denote the vector whose $i$-th component is $a^i$ by $(a^1, \ldots, a^n) \in \R^n$.
Given $k$ vectors $v^1, \ldots, v^k$, we denote the concatenated vector $[(v^1)^\top \ \cdots \ (v^k)^\top]^\top$ by $[v^1; \ldots; v^k]$.
A symmetric matrix-valued function $M : \R^m \to \R^{n \times n}$ is said to be uniformly positive definite if there exists $\epsilon>0$ such that $M(x) \succeq \epsilon I_n$ holds for any $x \in \R^m$.
For a scalar-valued function $f(x)$, where $x\in\R^n$, $\nabla f(x) \in \R^n$ denotes the gradient of $f$ at $x$.
For a vector-valued function $f(x) \in \R^m$, where $x\in\R^n$, $\frac{\partial f}{\partial x}$ denotes $\left[ \frac{\partial f}{\partial x_1} \  \cdots \ \frac{\partial f}{\partial x_n} \right] \in \R^{m\times n}$.
All functions are assumed to be differentiable an arbitrary number of times.

\section{PRELIMINARIES}

\subsection{Brief Introduction of Contraction Analysis \cite{hiro}}

Consider the following time-varying dynamics:
\begin{align}
  \dot{s}(t) = \phi(s(t),t) ,
  \label{dynamics}
\end{align}
where $s \in \R^n$ is the $n$-dimensional state and $\phi : \R^n \times \R \to \R^n$ is a differentiable function.
For notational simplicity we introduce the following operator:
\begin{align*}
    &\mathrm{L}[M, \phi, s, \gamma] (s,t)
    \\
    & \coloneqq \mathcal{L}_\phi M (s,t) + \frac{\partial M}{\partial t} (s,t) + 
    \left \langle M (s,t) \frac{\partial \phi}{\partial s} (s,t) \right \rangle 
    + \gamma M (s,t) ,
\end{align*}
where $M : \R^n \times \R \to \R^{n \times n}, (s,t) \mapsto M(s,t)$, and $\mathcal{L}_f M$ denotes Lie derivative of $M$ along vector field $f$, e.g., provided $x\in \R^n$ obeying $\dot{x}(t) = f(x(t),u(t),t)$,
\begin{align*}
    \mathcal{L}_f M (x,u,t) \coloneqq \sum_{i=1}^n f_i(x,u,t) \frac{\partial M}{\partial x_i} (x,t).
\end{align*}

\begin{definition}
    \label{def:contraction}
    Consider a uniformly positive definite function $M : \R^n \times \R \to \R^{n \times n}, (s,t) \mapsto M(s,t)$ and a constant $\gamma > 0$.
    Dynamics \eqref{dynamics} is said to be contracting with $M$ and $\gamma$ if the following matrix inequality holds for any $s\in \R^n$ and $t\in \R$:
    \begin{align}
        &\mathrm{L}[M, \phi, s, \gamma](s,t)  \preceq 0 .
        \label{condition_contraction}
    \end{align}
    In addition, we refer to the constant $\gamma$ as a contraction rate and the function $M$ as a contraction metric.
    Dynamics \eqref{dynamics} is simply said to be contracting if there exist $M$ and $\gamma$ with which dynamics \eqref{dynamics} is contracting.
    $\fin$
\end{definition}

The inequality \eqref{condition_contraction} implies the exponential stability of $\delta s \in \R^n$ obeying the following dynamics:
\begin{align}
    \dot{\delta s}(t) &= \frac{\partial \phi}{\partial s}(s(t),t) \delta s(t).
    \label{differential_dynamics}
    \end{align}
Indeed, $V_\delta(\delta s, t) \coloneqq \delta s^\top M(s(t),t) \delta s$ satisfies the following:
\begin{align}
    \delta s^\top  \mathrm{L}[M, \phi, s, \gamma](s(t), t)  \delta s
    = 
    \dot{V}_\delta(\delta s, t) + \gamma V_\delta(\delta s, t) .
    \label{V_delta_contraction}
\end{align}
Thus, \eqref{condition_contraction} indicates $\dot{V}_\delta + \gamma V_\delta \leq 0$, which implies that $V_\delta(\delta s, t)$ is a Lyapunov candidate for \eqref{differential_dynamics}.

If \eqref{dynamics} is contracting, then a path connecting arbitrary two points shrinks exponentially along their trajectories
(see Theorems 2.1 and 2.3 in \cite{hiro}).
Contraction theory allows us to discuss whether a suboptimally controlled trajectory converges to the optimally controlled trajectory.

\subsection{Continuation Method \cite{continuation, cgmres}}

Let the state and input of the system be $x\in \R^n$ and $u \in \R^m$, respectively.
Consider the following state space equation:
\begin{align}
  \dot{x}(t) = f(x(t), u(t), t),
  \label{ssm}
\end{align}
where $f : \R^n \times \R^m \times \R \to \R^n$.

To formulate optimal control problem (OCP), we discretize the equation \eqref{ssm}.
Denote the horizon of MPC by $h$ and the state and input at time $t^k$ by $x^k$ and $u^k$.
Let $x^0=x$ and discretize \eqref{ssm} as follows:
\begin{align}
  x^{k+1} &= f_\mathrm{d}(x^k, u^k, t^k) \ (k=0, \ldots, h-1). \label{eq:discrete_smm}
\end{align}
Typically, such a discretization is made by Euler difference, in which $t^k = t + k\varDelta t, \ f_\mathrm{d}(x^k, u^k, t^k) = x^k + f(x^k, u^k, t^k) \varDelta t$, where $\varDelta t$ is the time interval.

We denote the stage cost by $\ell(x,u,t)$ and the terminal cost by $\Phi(x,t)$, where the functions $\ell$ and $\Phi$ are smooth.
We also denote the design variable by $U \coloneqq [u^0 ; \ldots ; u^{h-1}] \in \R^{hm}$.
For simplicity we express the dependency of the trajectory on $(U,x,t)$ as $x^k(U,x,t)$.
Then, OCP is denoted as follows:
\begin{align}
  U^\ast(x,t) &= \mathrm{arg} \min_{U} ~ V(U,x,t)  ,
  \label{ocp}  \\
  V(U,x,t) &\coloneqq \Phi(x^h(U,x,t)) + \sum_{k=0}^{h-1} \ell (x^k(U,x,t), u^k) .
  \label{V_def}
\end{align}
First-order optimality condition claims that $U = U^\ast(x,t)$ is the solution of the following nonlinear equation system:
\begin{align}
  \zeta(U,x,t) &= 0,  \label{optimality}\\
  \zeta(U,x,t) &\coloneqq \nabla_U V(U,x,t) \in \R^{hm} . \label{zeta_def}
\end{align}
MPC utilizes $U^\ast(x,t)$ to obtain control input.
However, onboard calculation of $U^\ast(x,t)$ within the control frequency requires much computational resources.
Thus, the continuation method calculates $U$ in a suboptimal manner:
it updates $U(t)$ so that, instead of satisfying \eqref{optimality}, $\zeta(U(t), x(t), t) \to 0$ as $t\to \infty$.
Specifically, the continuation method embeds the time derivative of $\zeta(U,x,t)$ in the following \textit{virtual dynamics}:
\begin{align}
  \dot{z}(t) &= \eta(z(t), t)
  \label{virtual},
  \\
  z & = \zeta(U,x,t) ,
  \label{def_z}
\end{align}
where $\eta : \R^{hm} \times \R \to \R^{hm}$.
Note that $\eta$ in \eqref{virtual} is designed so that \eqref{virtual} is asymptotically stable, since $z = 0$ means that $U$ satisfies \eqref{optimality}.
A typical selection is $\eta(z, t) = - c z, \ c>0$.

Substituting \eqref{def_z} into \eqref{virtual} and expanding its time derivative lead us to the following  ordinary differential equation:
\begin{align}
    \dot{U}(t) &= 
    H(U(t),x(t),t)^{-1} b(U(t),x(t),t) , 
    \label{udot} \\
    H(U,x,t) &\coloneq \frac{\partial \zeta}{\partial U} (U,x,t),
    \label{def_H}  \\
    b(U,x,t) &\coloneqq 
    \eta_\zeta (U,x,t)
    - \frac{\partial \zeta}{\partial x}(U,x,t) f(x,u,t)
    - \frac{\partial \zeta}{\partial t} (U,x,t) ,
    \notag
\end{align}
where $\eta_\zeta (U,x,t) \coloneqq \eta(\zeta(U,x,t),t)$.
The control law of suboptimal MPC by the continuation method is $u(t) = \Pi_0 U(t)$, where $\Pi_0 \coloneqq [I_m\ 0\ \cdots \ 0] \in \R^{m\times hm}$ is a projection matrix from $U$ to $u^0$, and $U(t)$ is calculated by \eqref{udot}.
We should remark that the continuation method has a certain \textit{hierarchical structure}: $z(t)$ is determined by the virtual dynamics, and is converted into $U(t)$ via \eqref{def_z}.

This letter presumes the following assumption.

\begin{assumption}
  \label{ass:regularity}
  There exists a constant $\lambda_H>0$ such that $H$ in \eqref{def_H}, i.e., Hesse matrix of $V$ in \eqref{V_def} with respect to $U$, satisfies $H(U,x,t) \succeq \lambda_H I_{hm}$ for all $U\in \R^{hm}, x\in \R^n, t\in \R$.
  $\fin$
\end{assumption}
Assumption \ref{ass:regularity} ensures the regularity of $H$, which is vital for wellposedness of \eqref{udot}.
Note that Assumption \ref{ass:regularity} is sufficient for the second-order necessary condition, which claims that $H(U^\ast(x,t),x,t)\succeq0$ for all $x$ and $t$.
Also,
Assumption \ref{ass:regularity} is equivalent to strong convexity of $V$, which can be guaranteed by, for instance, input convex recurrent neural network \cite{icrnn};
see also Subsection IV.A in \cite{sparse_ocp}.

\begin{remark}
    \label{remark:feedforward}
    The term $- H^{-1} (-\frac{\partial \zeta}{\partial x}f -\frac{\partial \zeta}{\partial t})$ in \eqref{udot}
    becomes $\dot{U}^\ast(x(t),t)$ when $U(t) = U^\ast(x(t),t)$
    because we have
    \begin{align*}
        \frac{\mathrm{d} U^\ast(x(t), t)}{\mathrm{d} t} = \frac{\partial U^\ast}{\partial x} (x(t),t) f(x(t), u(t), t) + \frac{\partial U^\ast}{\partial t} (x(t),t), 
    \end{align*}
    and, from implicit differentiation of \eqref{def_z} at $U^\ast(x,t)$, we have
    \begin{align}
    \frac{\partial U^\ast}{\partial (x,t)} (U^\ast (x,t),x,t) = - (H^\ast)^{-1} \frac{\partial \zeta}{\partial (x,t)} (U^\ast (x,t), x, t)
    ,
    \label{uastdot}
    \end{align}
    where $H^\ast \coloneqq H(U^\ast(x, t),x,t)$.
    Therefore, \eqref{udot} admits $U(t) \equiv U^\ast(x(t),t)$.
    See also Subsection II.D in \cite{tvo-survey} and Theorem 3 in \cite{tvo_contraction}.
    Because contraction theory states the convergence of trajectories, it is relevant to our study whether the suboptimal MPC algorithm admits $U(t)\equiv U^\ast(x(t),t)$.
\end{remark}

\begin{remark}
We impose no inequality constraints on OCP \eqref{ocp} throughout this letter because the control law includes switching and hybrid system analysis is required.
Extension of our results to constrained MPC is left for future work.
$\fin$
\end{remark}

\subsection{Problem Settings for Contraction Analysis}
\label{subsection:problem_setting}

The closed-loop system is formulated as follows:
\begin{align}
    \dot{s}(t) &= \phi(s(t), t)  , \label{dynamics_s}\\
    s &\coloneqq [x; U] \in \R^{n+hm} , \\
    \phi(s, t) &\coloneqq \left[ f(x, \Pi_0 U, t) ; H^{-1}(U,x,t) b(U,x,t) \right] . \notag
\end{align}
In addition, we denote the optimally controlled state by $x^\ast$, which obeys the following ordinary differential equation:
\begin{align}
    \dot{x}^\ast(t) = f(x^\ast(t), \Pi_0 U^\ast(x^\ast(t),t),t)  .
    \label{dynamics_optimal}
\end{align}
The matrix inequality of the contraction of \eqref{dynamics_s} is as follows:
\begin{align}
    \mathrm{L}[M, \phi, s, \gamma](s,t)  \preceq 0 .
    \label{goal}
\end{align}
The goal in this letter is to discuss the convergence of the trajectory $[x(t); U(t)]$ to the optimally controlled trajectory $[x^\ast(t); U^\ast(x^\ast(t),t)]$ by analyzing \eqref{goal}.
The key is the choice of $M(s,t)$, since its random choice could make the inequality \eqref{goal} intractable.
To avoid this, we reflect the aforementioned hierarchical structure of the continuation method to $M$.

\section{MAIN RESULT}

This section provides the main results of this letter.
Subsection A proposes a contraction metric of the continuation method.
Using our proposed metric, Subsection B derives a pair of matrix inequalities that derive \eqref{goal}.
Subsection C provides the proof of our main theorem. 
Since contraction theory discusses convergence not to an equilibrium but to a trajectory, our results in this section can be applied to other MPCs, including economic MPC.

\subsection{Contraction Metric of Continuation Method}

We propose the following metric $M$:
\begin{gather}
  M(s, t) = M_P(x,t) + \kappa M_Q(s, t) \label{metric}  , \\
  M_P(x,t) \coloneqq \begin{bmatrix} P(x,t) & 0 \\ 0 & 0 \end{bmatrix} ,\ \ 
  M_Q(s,t) \coloneqq \left( \frac{\partial \zeta}{\partial s} \right)^\top Q_\zeta(s,t) \frac{\partial \zeta}{\partial s}  
  ,
  \notag
\end{gather}
where $Q_\zeta : \R^{hm} \times \R \to \R^{hm \times hm}$ is defined as $Q_\zeta(s,t) \coloneqq Q(\zeta(s,t), t)$, $P : \R^n \times \R \to \R^{n\times n}, \ Q: \R^{hm} \times \R \to \R^{hm \times hm}$ are uniformly positive definite functions, $\kappa>0$ is a constant, and the arguments $(s,t)$ of $\zeta$ is omitted.
$M_Q$ plays a role of a contraction metric of \eqref{virtual}, while $M_P$ evaluates the contraction of variable $x$ (see also Theorem \ref{thm:matrix_inequality_schur}).

The proposed function $M$ is indeed a metric;
$M$ is positive definite under Assumption \ref{ass:regularity}, and the uniformity holds under the Lipschitzness of $\zeta$ with respect to $x$.

\begin{assumption}
\label{ass:bounded_zetax}
There exists a constant $c_x^\zeta > 0$ such that
    \begin{align*}
        \left\| \frac{\partial \zeta}{\partial x} (U, x, t) \right\| \leq c_x^\zeta  \ \ \  \forall U\in \R^{hm}, x\in\R^n, t\in \R .
    \end{align*}
\end{assumption}

\begin{lemma}
    Under Assumptions \ref{ass:regularity} and \ref{ass:bounded_zetax}, $M$ in \eqref{metric} is uniformly positive definite.
\end{lemma}

\begin{proof}
Let $P\succeq \epsilon_P I_n, \ Q \succeq \epsilon_Q I_{hm}$, where $\epsilon_P, \epsilon_Q>0$.
Take $\epsilon \in \R$ that satisfies
\begin{align}
    0 < \epsilon < \kappa \epsilon_Q \lambda_H^2, \ \frac{(c_x^\zeta)^2 \kappa \epsilon_Q}{\kappa \epsilon_Q \lambda_H^2 - \epsilon} \epsilon + \epsilon \leq \epsilon_P . \label{def_epsilon}
\end{align}
$M-\epsilon I_{n+hm}$ is determined as follows:
\begin{align*}
    M - \epsilon I_{n+hm}
    = \begin{bmatrix}
        P - \epsilon I_n + \kappa \left( \frac{\partial \zeta}{\partial x} \right)^\top Q_\zeta \frac{\partial \zeta}{\partial x}
        & \kappa \left( \frac{\partial \zeta}{\partial x} \right)^\top Q_\zeta H
        \\
        \kappa H Q_\zeta \frac{\partial \zeta}{\partial x}
        & \kappa H Q_\zeta H - \epsilon I_{hm}
    \end{bmatrix}  .
\end{align*}
Owing to \eqref{def_epsilon}, we obtain $\kappa H Q H - \epsilon I_{hm} \succeq (\kappa \lambda_H^2 \epsilon_Q - \epsilon) I_{hm} \succ 0_{hm}$.
We denote the Schur complement with respect to the lower right block by $S$.
Then, we obtain
\begin{align*}
    S &= P -\epsilon I_n + \kappa \left( \frac{\partial \zeta}{\partial x} \right)^\top Q_\zeta \frac{\partial \zeta}{\partial x}  \\
    & \hspace{10pt} - \kappa \left( \frac{\partial \zeta}{\partial x} \right)^\top Q_\zeta
    \left( Q_\zeta - \frac{\epsilon}{\kappa} H^{-2} \right)^{-1}
    Q_\zeta \frac{\partial \zeta}{\partial x} \\
    &= P -\epsilon I_n
    + \kappa \left( \frac{\partial \zeta}{\partial x} \right)^\top
    \left( Q_\zeta^{-1} - \frac{\kappa}{\epsilon} H^2 \right)^{-1} \frac{\partial \zeta}{\partial x} \\
    & \succeq \left( \epsilon_P -\epsilon - \frac{(c_x^\zeta)^2 \kappa \epsilon_Q \epsilon}{\kappa \epsilon_Q \lambda_H^2 - \epsilon} \right) I_n ,
\end{align*}
where the second equality originates from the following Woodbury formula:
\begin{align*}
    \left( Q_\zeta - \frac{\epsilon}{\kappa} H^{-2} \right)^{-1} = Q_\zeta^{-1} - Q_\zeta^{-1} \left( Q_\zeta^{-1} - \frac{\kappa}{\epsilon} H^{2} \right)^{-1} Q_\zeta^{-1} ,
\end{align*}
and the final inequality originates from the following fact:
\begin{align*}
    Q_\zeta^{-1} - \frac{\kappa}{\epsilon} H^{2} \preceq \left(\frac{1}{\epsilon_Q} - \frac{\kappa \lambda_H^2}{\epsilon}\right) I_{hm} \preceq 0_{hm} .
\end{align*}
Therefore, owing to \eqref{def_epsilon}, we obtain $S \succeq 0$, which implies $M$ is uniformly positive definite.
\end{proof}

\begin{remark}
\label{remark:ass_zetax}
    Assumption \ref{ass:bounded_zetax} holds true if norms of $\nabla_x \ell$, $\nabla_x \Phi$, and the following $(n+3)$ matrices are upper bounded:
    \begin{align*}
        \frac{\partial f_\mathrm{d}}{\partial (x,u)}, \ \ 
        \frac{\partial (\nabla_x (f_\mathrm{d})_i)}{\partial (x,u)} \ (i=1,\ldots, n), \ \ 
        \frac{\partial (\nabla_x \ell)}{\partial (x,u)}, \ \ 
        \frac{\partial (\nabla_x \Phi)}{\partial x} .
    \end{align*}
\end{remark}

\subsection{Matrix Inequalities for Contraction}

We reformulate the dynamics \eqref{ssm} in terms of residual of optimal control and obtain the following dynamics:
\begin{align}
  \dot{x}(t) &= f_\mathrm{r}(x(t), u_\mathrm{r}(t), t) , \label{ssm_v} \\
  u_\mathrm{r} &\coloneqq u - \Pi_0 U^\ast(x,t) , \label{def_ur}\\
  f_\mathrm{r}(x,u_\mathrm{r},t) &\coloneqq f(x, \Pi_0 U^\ast(x, t) + u_\mathrm{r}, t)  .
\end{align}

To obtain a pair of matrix inequalities that characterize contraction of the closed-loop system, we impose the following assumption:
\begin{assumption}
    \label{ass:bounded_fu}
    There exists a constant $c_u^f > 0$ such that
    \begin{align*}
        \left\| \frac{\partial f}{\partial u} (x,u,t) \right\| \leq c_u^f  \ \ \ \forall x\in R^n, u\in\R^m, t\in\R,
    \end{align*}
    i.e., $f$ in \eqref{ssm} has a Lipschitz constant $c_u^f$ with respect to $u$.
    $\fin$
\end{assumption}
Assumption \ref{ass:bounded_fu} ensures wellposedness of the problem setting because it indicates that the effect of $u$ on $\dot{x}$ is finite.

\begin{theorem}
    \label{thm:matrix_inequality_schur}
    Suppose Assumptions \ref{ass:regularity} and \ref{ass:bounded_fu} hold.
    Suppose uniformly positive definite matrices $P,Q$ and constants $\beta_x, \beta_z>0$ satisfy the following two matrix inequalities:
    \begin{gather}
        \begin{aligned}
        \mathrm{L}[P, f_\mathrm{r}, x, \beta_x](x,u_\mathrm{r},t) + 
        \left\langle P_{xU}(x,u_\mathrm{r},t)K(U,x,t) \right\rangle \preceq 0&   
        \\
        \forall U \in \R^{hm},\ x\in \R^n,\ t \in \R&,
        \end{aligned}
        \label{design_matrix_inequality_P}
        \\
        \mathrm{L}[Q, \eta, z, \beta_z](z,t) \preceq 0   \ \ \ \ \forall z \in \R^{hm},\ t\in \R,
        \label{design_matrix_inequality_Q}
        \\
        P_{xU} (x,u_\mathrm{r},t) \coloneqq P(x,t) \frac{\partial f_\mathrm{r}}{\partial u_\mathrm{r}} (x,u_\mathrm{r},t) \Pi_0 , 
        \label{def_Pbar}
        \\
        K(U,x,t) \coloneqq - H^{-1}(U,x,t) \frac{\partial \zeta}{\partial x}(U,x,t) - \frac{\partial U^\ast}{\partial x}(x,t) ,
        \label{def_K}
    \end{gather}
    where $u_\mathrm{r}$ in \eqref{design_matrix_inequality_P} is defined using $U$ and $x$ as in \eqref{def_ur}.
    Then, there exists $\gamma>0$ such that \eqref{goal} holds true for $M$ in \eqref{metric}.
\end{theorem}

The proof of Theorem \ref{thm:matrix_inequality_schur} is presented in Subsection \ref{subsection:proof}.
Note that \eqref{design_matrix_inequality_Q} implies that \eqref{virtual} is contracting with $Q$ and $\beta_z$ and that the term $\mathrm{L}[P, f_\mathrm{r}, x, \gamma]$ in \eqref{design_matrix_inequality_P} expresses the contraction of \eqref{ssm_v} with $P$ and $\beta_x$.
Moreover, \eqref{uastdot} implies $K=0$ when $U=U^\ast(x,t)$.
Therefore, the term $\left\langle P_{xU}K \right\rangle$ indicates the impact of the suboptimality on the contraction of the optimally controlled dynamics.
By binding $\left\langle P_{xU}K \right\rangle$ with $P$, we obtain the following corollary.

\begin{corollary}
    \label{cor:matrix_inequality_reduced}
    Suppose that Assumptions \ref{ass:regularity} and \ref{ass:bounded_fu} hold and that there exist uniformly positive definite matrices $P,Q$ and constants $\beta_x, \beta_z, \beta_\mathrm{p}$ with $\beta_x > \beta_\mathrm{p} > 0$, $\beta_z > 0$ such that \eqref{design_matrix_inequality_Q} and the following two matrix inequalities hold:
    \begin{align}
        \langle P_{xU}(x,u,t)  K(U,x,t) \rangle - \beta_\mathrm{p} P(x,t) & \preceq 0 \ \ \ \forall U,x,t,
        \label{design_matrix_inequality_GK}
        \\
        \mathrm{L}[P, f_\mathrm{r}, x, \beta_x + \beta_\mathrm{p}](x,u_\mathrm{r},t) &\preceq 0   \ \ \ \forall x, u_\mathrm{r}, t ,
        \label{design_matrix_inequality_Pxxopt}
    \end{align}
    where $P_{xU}$ is defined as in \eqref{def_Pbar}.
    Then, there exists $\gamma>0$ such that \eqref{goal} holds for $M$ in \eqref{metric}.
    Moreover, the closed-loop system \eqref{dynamics_s} is contracting under Assumption \ref{ass:bounded_zetax}.
    Furthermore, if $\eta(0,t)=0$ for any $t \in \R$ in \eqref{virtual}, then $x(t)$ determined by \eqref{dynamics_s} converge to $x(t)$ determined by \eqref{dynamics_optimal}.
\end{corollary}

\begin{proof}
The existence of $\gamma$ and the contraction of the closed-loop system is implied by Theorem \ref{thm:matrix_inequality_schur} because \eqref{design_matrix_inequality_GK} and \eqref{design_matrix_inequality_Pxxopt} implies \eqref{design_matrix_inequality_P}.
Also, from $\eta(0,t)=0$, \eqref{virtual} admits $z(t)\equiv 0$, which means that \eqref{dynamics_s} admits $U(t) \equiv U^\ast(x(t),t)$.
Therefore, the contraction of \eqref{dynamics_s} implies that $[x(t); U(t)]$ obeying \eqref{dynamics_s} converges to $[x(t); U^\ast(x(t),t)]$, where $x(t)$ is the solution of \eqref{dynamics_optimal}.
\end{proof}

The examined space of \eqref{design_matrix_inequality_Pxxopt} is smaller than that of \eqref{design_matrix_inequality_P}.
Besides, under Assumptions \ref{ass:regularity}, \ref{ass:bounded_zetax}, and \ref{ass:bounded_fu},
\eqref{design_matrix_inequality_GK} holds if $2 c_u^f c_x^\zeta p_\mathrm{max} \leq \beta_\mathrm{p} p_\mathrm{min} \lambda_H$ holds, where $p_\mathrm{min} I_n \preceq P(x,t) \preceq p_\mathrm{max} I_n$ for all $x \in \R^n,t \in \R$.
The matrix inequality \eqref{design_matrix_inequality_Pxxopt} implies the contraction of the optimally controlled closed-loop system.
MPC design procedure that guarantees \eqref{design_matrix_inequality_Pxxopt} is left for future work.
One remedy is to directly impose \eqref{design_matrix_inequality_Pxxopt} on OCP as in \cite{wang2017}.

\subsection{Proof of Theorem \ref{thm:matrix_inequality_schur}}
\label{subsection:proof}

\begin{lemma}
    \label{thm:matrix_inequality}
    For $M$ in \eqref{metric}, the left hand side of \eqref{goal} is reformulated as follows:
    \begin{align}
        \mathrm{L}[M, \phi, s, \gamma]
        &=
        \begin{bmatrix} \mathrm{L}^{(\gamma)}_x & \bar{P}_{xU} \\ \bar{P}_{xU}^\top & 0 \end{bmatrix}
        + \kappa \left( \frac{\partial \zeta}{\partial s} \right)^\top \mathrm{L}^{(\gamma)}_\zeta \frac{\partial \zeta}{\partial s} ,
        \label{matrix_contraction_continuation}
        \\
        \mathrm{L}^{(\gamma)}_x (x,u,t) &\coloneqq \mathrm{L}[P, f, x, \gamma](x,u,t), 
        \label{def_Pxx}
        \\
        \mathrm{L}^{(\gamma)}_\zeta (s,t) &\coloneqq \mathrm{L}^{(\gamma)}_z (\zeta(s,t),t) ,
        \label{def_Lzeta}
        \\
        \mathrm{L}^{(\gamma)}_z (z,t) &\coloneqq \mathrm{L}[Q, \eta, z, \gamma](z,t),
        \label{def_Lz}
        \\
        \bar{P}_{xU}(x,u,t) &\coloneqq P_{xU}(x, u - \Pi_0 U^\ast(x,t), t)
        ,
        \label{def_PxU}
    \end{align}
    where $G$ is defined in \eqref{def_Pbar}, and we omit the arguments $(s,t)$ of $\mathrm{L}[M, \phi, s, \gamma]$, $\zeta$, and $\mathrm{L}^{(\gamma)}_\zeta$ and $(x,u,t)$ of $\mathrm{L}^{(\gamma)}_x$ and $\bar{P}_{xU}$ in \eqref{matrix_contraction_continuation} for notational simplicity.
\end{lemma}

\begin{proof}
    In this proof we omit arguments $(s,t)$ of $\phi$, $\zeta$, $M_Q$, and $Q_\zeta$ and $(z,t)$ of $\eta$ unless it's confusing.
    Straightforward algebraic calculation leads us to the following relation:
    \begin{align*}
    \mathrm{L}[M_P, \phi, s, \gamma](s,t)
    = \begin{bmatrix} \mathrm{L}^{(\gamma)}_x (s,t) & \bar{P}_{xU}(x,u,t) \\ \bar{P}_{xU}(x,u,t)^\top & 0 \end{bmatrix} .
    \end{align*}
    
    Multiplying $\frac{\partial \zeta}{\partial U}$ to \eqref{dynamics_s} derives the following relation:
    \begin{align}
        \eta(\zeta(s,t),t) = \mathcal{L}_{\phi} \zeta (s,t) + \frac{\partial \zeta}{\partial t} (s,t)
        .
        \label{eq:zetadot}
    \end{align}
    Differentiating \eqref{eq:zetadot} with respect to $s$ derives the following:
    \begin{align*}
        \left. \frac{\partial \eta}{\partial z} \right|_{z=\zeta}  \frac{\partial \zeta}{\partial s} 
        &= \frac{\partial \zeta}{\partial s} \frac{\partial \phi}{\partial s} 
        + \left [ \frac{\partial}{\partial s} \left .\left( \sum_{i=1}^{n+hm} v_i \frac{\partial \zeta}{\partial s_i} \right) \right ] \right|_{v=\phi}
        + \frac{\partial}{\partial s} \frac{\partial \zeta}{\partial t}
        \\
        &=
        \frac{\partial \zeta}{\partial s} \frac{\partial \phi}{\partial s} 
        + 
        \sum_{i=1}^{n+hm} \phi_i \frac{\partial}{\partial s_i} \frac{\partial \zeta}{\partial s} + \frac{\partial}{\partial t} \frac{\partial \zeta}{\partial s} 
    \end{align*}
    where the second equality comes from the commutability of derivative operator.
    In addition, \eqref{eq:zetadot} implies the followings:
    \begin{align*}
        \mathcal{L}_{\eta} Q(\zeta(s,t), t) + \frac{\partial Q}{\partial t}(\zeta(s,t),t)
        =
        \mathcal{L}_{\phi} Q_\zeta(s,t) + \frac{\partial Q_\zeta}{\partial t}(s,t) .
    \end{align*}
    Furthermore, we have
    \begin{align*}
    \mathcal{L}_{\phi} M_Q + \frac{\partial M_Q}{\partial t}
    &= 
    \left( \frac{\partial \zeta}{\partial s} \right)^\top \left(\mathcal{L}_{\phi}Q_\zeta + \frac{\partial Q_\zeta}{\partial t}\right) \frac{\partial \zeta}{\partial s}
    \\
    & \hspace{-20pt} + 
    \left \langle
    \left( \frac{\partial \zeta}{\partial s} \right)^\top Q_\zeta \cdot \left( \sum_{i=1}^{n+hm} \phi_i \frac{\partial}{\partial s_i} \frac{\partial \zeta}{\partial s} + \frac{\partial}{\partial t} \frac{\partial \zeta}{\partial s} \right)
    \right \rangle 
    .
    \end{align*}
    For simplicity, we denote $\frac{\partial \zeta}{\partial s}(s,t)$ by $\zeta_s(s,t)$.
    Then, by a simple algebraic calculation, we obtain
    \begin{align*}
        &\zeta_s^\top
        \left(\mathcal{L}_\eta Q |_{z=\zeta} + \left . \frac{\partial Q}{\partial t} \right|_{z=\zeta} + 
        \left \langle Q_\zeta \left . \frac{\partial \eta}{\partial z} \right|_{z=\zeta} \right \rangle \right)
        \zeta_s
        \\
        &=
        \zeta_s^\top
        \left(\mathcal{L}_\phi Q_\zeta + \frac{\partial Q_\zeta}{\partial t} \right)
        \zeta_s
         + 
        \left \langle \zeta_s^\top Q_\zeta \left . \frac{\partial \eta}{\partial z} \right|_{z=\zeta} \zeta_s \right \rangle
        \\
        &=
        \mathcal{L}_{\phi} M_Q + \frac{\partial M_Q}{\partial t}
         + 
        \left \langle \zeta_s^\top Q_\zeta \zeta_s \frac{\partial \phi}{\partial s} \right \rangle  ,
    \end{align*}
    where the arguments $(s,t)$ are omitted.
    This relation implies $\zeta_s (s,t)^\top \mathrm{L}^{(\gamma)}_\zeta (s,t) \zeta_s (s,t) = \mathrm{L}[M_Q, \phi, s, \gamma] (s,t)$.
    Therefore, \eqref{matrix_contraction_continuation} holds true.
\end{proof}

% Subsequently, we present the following lemma:
\begin{lemma}
    \label{prop:matrix_inequality}
    Suppose Assumptions \ref{ass:regularity} and \ref{ass:bounded_fu} hold.
    Suppose there exist uniformly positive definite matrices $P,Q$ and constants $\beta_x, \beta_z>0$ such that \eqref{design_matrix_inequality_Q} and the following matrix inequality hold:
    \begin{align}
        \mathrm{L}^{(\beta_x)}_x (x,u,t)
        - \left \langle \bar{P}_{xU} (x,u,t) H^{-1} (U,x,t) \frac{\partial \zeta}{\partial x} (U,x,t) \right \rangle
        \notag
        &\preceq 0 
        \\
        \forall U \in \R^{hm}, \ x\in \R^n, \ t\in \R,&
        \label{design_matrix_inequality}
    \end{align}
    where $u=\Pi_0 U$, and $\bar{P}_{xU}$ and $\mathrm{L}^{(\cdot)}_x$ are defined in \eqref{def_PxU} and \eqref{def_Pxx}, respectively.
    Then, there exist $M$ and $\gamma$ satisfying \eqref{goal}.
\end{lemma}

\begin{proof}
Denote the minimum eigenvalues of $P$ and $Q$ by $p_\mathrm{min}$ and $q_\mathrm{min}$, and the maximum eigenvalue of $P$ by $p_\mathrm{max}$.
Take $\gamma$ satisfying $0<\gamma<\min\{\beta_x,\beta_z\}$ and select $\kappa$ as follow:
\begin{align*}
    \kappa = \frac{(c_u^f)^2 p_\mathrm{max}^2}{\lambda_H^2 (\beta_x-\gamma) (\beta_z-\gamma) p_\mathrm{min} q_\mathrm{min}} .
\end{align*}
Owing to Lemma \ref{thm:matrix_inequality}, the goal in this proof is to show that the right-hand side of \eqref{matrix_contraction_continuation} with $\gamma$ we took and $P$ and $Q$ in the assumption is negative semidefinite for all $s,t$.

The lower right block on the right-hand side of \eqref{matrix_contraction_continuation} is $H \mathrm{L}^{(\gamma)}_\zeta H$.
The inequality \eqref{design_matrix_inequality_Q} indicates $\mathrm{L}^{(\beta_z)}_\zeta \preceq 0$, and from \eqref{def_Lz} we have $\mathrm{L}^{(\gamma)}_\zeta = \mathrm{L}^{(\beta_z)}_\zeta - (\beta_z-\gamma) Q_\zeta \prec \mathrm{L}^{(\beta_z)}_\zeta \preceq 0$.
Therefore, considering Assumption \ref{ass:regularity}, we have $H \mathrm{L}^{(\gamma)}_\zeta H \prec 0 $.
Denote the Schur complement of $H \mathrm{L}^{(\gamma)}_\zeta H$ in the right-hand side of \eqref{matrix_contraction_continuation} by $S_2$.
Taking into account the fact that
\begin{align*}
    \left( \frac{\partial \zeta}{\partial x} \right)^\top \mathrm{L}^{(\gamma)}_\zeta \frac{\partial \zeta}{\partial x} = \left( \frac{\partial \zeta}{\partial x} \right)^\top \mathrm{L}^{(\gamma)}_\zeta H \left ( H \mathrm{L}^{(\gamma)}_\zeta H \right)^{-1} H \mathrm{L}^{(\gamma)}_\zeta \frac{\partial \zeta}{\partial x}  ,
\end{align*}
we obtain, from simple algebraic calculation,
\begin{align*}
    S_2
    &= \mathrm{L}^{(\gamma)}_x
    - \left \langle \bar{P}_{xU} H^{-1} \frac{\partial \zeta}{\partial x} \right \rangle
    - \kappa^{-1} N , \\
    N
    & \coloneqq \bar{P}_{xU}
    H^{-1}
    \left( \mathrm{L}^{(\gamma)}_\zeta \right)^{-1}
    H^{-1}
    \bar{P}_{xU}^\top .
\end{align*}
From \eqref{design_matrix_inequality}, we have $S_2 \preceq \mathrm{L}^{(\gamma)}_x - \mathrm{L}^{(\beta_x)}_x - \kappa^{-1} N$.
Since \eqref{design_matrix_inequality_Q} implies $\mathrm{L}^{(\gamma)}_\zeta \preceq 0$, we obtain the followings:
\begin{gather*}
    \mathrm{L}^{(\gamma)}_x - \mathrm{L}^{(\beta_x)}_x = - (\beta_x-\gamma)P \preceq - (\beta_x-\gamma) p_\mathrm{min} I_{hm}, 
    \\
    \mathrm{L}^{(\gamma)}_\zeta = \mathrm{L}^{(\beta_z)}_\zeta - (\beta_z-\gamma)Q \preceq - (\beta_z-\gamma) Q \preceq - (\beta_z-\gamma) q_\mathrm{min} I_n .
\end{gather*}
Therefore, under Assumptions \ref{ass:regularity} and \ref{ass:bounded_fu}, we obtain
\begin{align*}
    -\frac{(c_u^f)^2 p_\mathrm{max}^2}{\lambda_H^2 (\beta_z-\gamma)q_\mathrm{min}} I_n \preceq N \prec 0
\end{align*}
and thus $-(\beta_x-\gamma) p_\mathrm{min} I_n \preceq \kappa^{-1} N$
from the choice of $\kappa$.
Therefore, we have $S_2 \preceq 0$.
\end{proof}

The rest is the proof of Theorem \ref{thm:matrix_inequality_schur}.

\begin{proof}
Lemma \ref{prop:matrix_inequality} derives the claim by taking the following relation into account:
\begin{align*}
    \frac{\partial f_\mathrm{r}}{\partial x}(x, u_\mathrm{r}, t)
    = \frac{\partial f}{\partial x}(x,u,t)
    - \frac{\partial f}{\partial u}(x,u,t) \Pi_0 \frac{\partial U^\ast}{\partial x} (x,t).
\end{align*}
\end{proof}

\section{Numerical Verification}

\subsection{Problem Setting}

Consider the plant model \eqref{ssm} with $f$ determined as follows:
\begin{gather*}
  f(x,u,t) = A(t) x + B u + r(x) + w(t), 
  \\
  r(x) \coloneqq 0.2 \times 
  (0, f_\mathrm{act}(x_1), 0, f_\mathrm{act}(x_3))
  \\
  f_\mathrm{act}(a) \coloneqq \log(\exp(a)+1) - \log(2)
  \\
  w(t) = 0.3 \times (0, \mathrm{cs}(t), 0, \mathrm{cs}(t)) .
  \\
  A(t) \coloneqq
  \begin{bmatrix}
    -0.5& 0& 0& 0\\
    1& -1+\frac{\mathrm{cs}(t)}{2}& 0& 0\\
    0& 1& -1+\frac{\mathrm{cs}(t)}{2}& 0\\
    0& 0& 1& -1+\frac{\mathrm{sn}(t)}{2}
  \end{bmatrix} ,
  \\
  \mathrm{cs}(t) = \cos(2t/\pi), \ \mathrm{sn}(t) = \sin(2t/\pi),
  \\
  B \coloneqq
  \begin{bmatrix}
    1 & 0 & 0 & 0\\
    0 & 0 & 1 & 0
  \end{bmatrix} ^\top .
\end{gather*}
Regarding OCP \eqref{ocp} for MPC, the difference scheme is selected as the forward difference with $\varDelta t = 0.5$.
The stage cost $\ell$ and the terminal cost $\Phi$ is selected as follows:
\begin{align*}
  \ell(x^k,u^k,t) = 0.5 \| x^k \|^2_2 + \| u^k \|^2_2 , \
  \Phi(x^h,t) = \| x^h \|^2_2  ,
\end{align*}
where $h=3$.
We select the function $\eta$ in \eqref{virtual} as $\eta_i(z,t) = -0.2z_i -0.05z_i^3, \ i=1,\ldots,6$.

\subsection{Numerical Examination of Assumptions in Corollary \ref{cor:matrix_inequality_reduced}}

\begin{figure}[t!]
  \centering
  \includegraphics[width=1.0\linewidth]{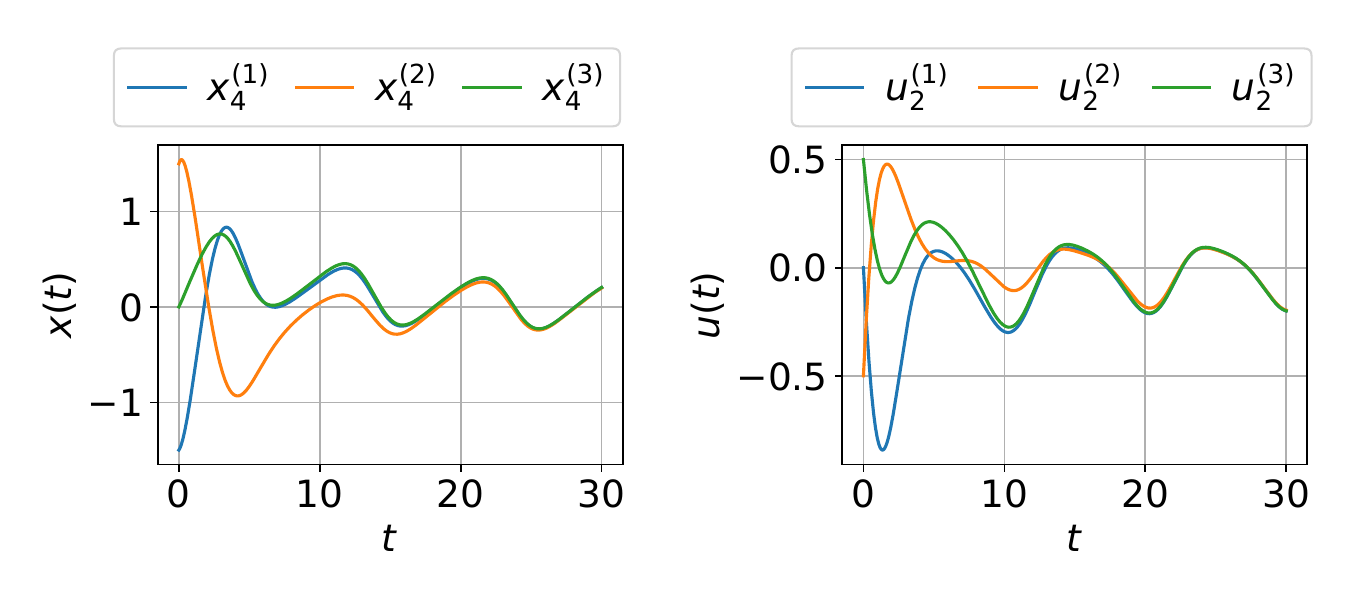}
  \caption{Trajectories of $x_4$ and $u_2$ under three different initial condition.  \label{fig:xu}}
\end{figure}

\begin{figure}[t!]
  \centering
  \includegraphics[width=0.7\linewidth]{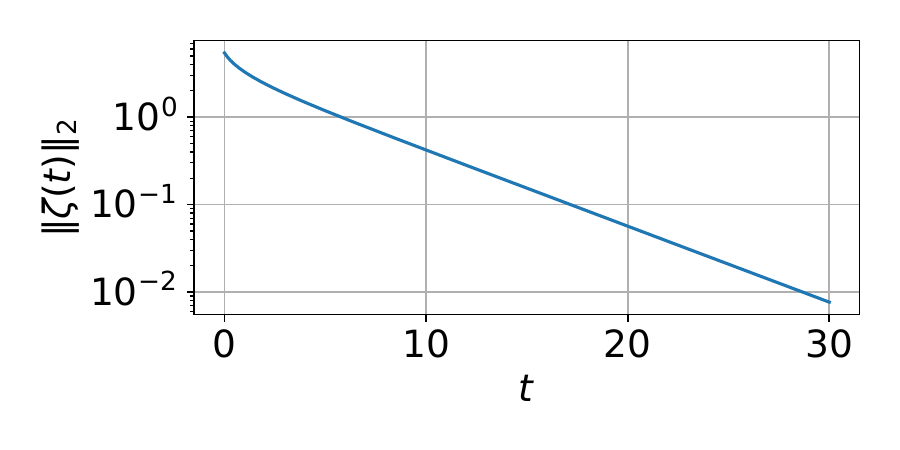}
  \caption{Norm of $\zeta(U^{(1)}(t), x^{(1)}(t), t)$.  \label{fig:fnorm}}
\end{figure}

We check if the assumptions of Corollary \ref{cor:matrix_inequality_reduced} holds.
According to Remark \ref{remark:ass_zetax}, Assumption \ref{ass:bounded_zetax} holds.
We select $P(x,t)=I_4$, $Q(z,t)=I_6$, and $\beta_z=0.4$.
Under this selection, \eqref{design_matrix_inequality_Q} holds.
Moreover, $\eta(0,t)=0$ holds for any $t$, and Assumption \ref{ass:bounded_fu} holds true because $B$ is constant.
Therefore, the remaining assumptions are 
Assumption \ref{ass:regularity},
\eqref{design_matrix_inequality_GK},
and
\eqref{design_matrix_inequality_Pxxopt}.

We examine \eqref{design_matrix_inequality_Pxxopt} by meshing the variables.
Because \eqref{design_matrix_inequality_Pxxopt} gets independent on $u_\mathrm{r}$ in this case, we examine \eqref{design_matrix_inequality_Pxxopt} by meshing only $x$ and $t$.
Consider 21 points on $[-2,2]$ of each element of $x\in \R^4$ at intervals of 0.2, and 40 points on $[0, 3.9]$ of $t\in \R$ at intervals of 0.1 (hence, 7,779,240 points in total).
The maximum eigenvalue of $\mathrm{L}[P, f_\mathrm{r}, x, \beta_x + \beta_\mathrm{p}]$ where $\beta_x = 0.1, \beta_\mathrm{p} = 0.032$ on all points is $-0.105$; thus, the inequality \eqref{design_matrix_inequality_Pxxopt} holds at all checked points.

Next, we examine Assumption \ref{ass:regularity} and \eqref{design_matrix_inequality_GK}.
Consider five points on $[-1.6, 1.6]$ of $x\in \R^4$ at intervals of 0.8, four points on $[0.5, 3.5]$ of $t\in \R$ at intervals of 1, and seven points on $[-0.6, 0.6]$ of each element of $U\in \R^6$ at intervals of 0.2 (hence, 294,122,500 points in total).
The eigenvalue of the Jacobi matrix of $\zeta$ with respect to $U$ on all points is at least $2.11$; thus, Assumption \ref{ass:regularity} holds at all checked points.
In addition, the norm of $\langle P_{xU}K\rangle$ on all points is at most $0.0159$; thus, \eqref{design_matrix_inequality_GK} holds for $\beta_\mathrm{p}=0.032$ at all checked points.

As far as we checked, all of the assumptions in Corollary \ref{cor:matrix_inequality_reduced} hold.
Thus, the path between any two trajectories shrinks, and therefore the trajectories by the continuation method are expected to converge to the optimally controlled trajectory.

\subsection{Results of Control}

We calculate $s(t)$ according to \eqref{dynamics_s} on three different initial conditions.
We denote those three trajectories by $s^{(1)}$, $s^{(2)}$, and $ s^{(3)}$.
We also denote $x$ and $U$ of $ s^{(i)}$ by $x^{(i)}$ and $U^{(i)}$ for $i=1,2,3$.
Each different initial condition is $x^{(1)}(0) = (1.5, 1.5, -1.5, -1.5)$, $x^{(2)}(0) = (-1.5, -1.5, 1.5, 1.5)$, and $x^{(3)}(0) = 0$ for $x$, and $U^{(1)}(0)=0$, $U^{(2)}(0)=(-0.5, -0.5, 0, 0, 0, 0)$, and $U^{(3)}(0)=(0.5, 0.5, 0, 0, 0, 0)$ for $U$.
The trajectories of $x_4$ and $u_2$ are depicted in Fig.~\ref{fig:xu}, which shows that the trajectories in all cases converge to the same trajectory.
The value of $\zeta$ of $s^{(1)}(t)$ is depicted in Fig.~\ref{fig:fnorm}, which indicates $U^{(1)}$ (therefore $U^{(2)}$ and $U^{(3)}$ as well) asymptotically converges to the optimally controlled trajectory.

\subsection{Numerical Verification of Lemma \ref{thm:matrix_inequality}}

Lemma \ref{thm:matrix_inequality} (see also \eqref{V_delta_contraction}) claims $\dot{V}_\delta(t) + \gamma V_\delta(t) = d$, where
\begin{gather*}
    V_\delta(t) \coloneqq \delta s (t)^\top M(s(t),t) \delta s (t),
    \\
    d \coloneqq \delta s^\top \left (
    \begin{bmatrix} \mathrm{L}^{(\gamma)}_x & P G \\ G^\top P & 0 \end{bmatrix}
    + \kappa \left( \frac{\partial \zeta}{\partial s} \right)^\top \mathrm{L}^{(\gamma)}_\zeta \frac{\partial \zeta}{\partial s}
    \right ) \delta s .
\end{gather*}
In this subsection we numerically verify this relation.

First, we select the initial condition as $s(0) = [x^{(1)}(0); U^{(1)}(0)]$.
Second, we generate 100 perturbations from a uniform distribution on a section $[-\varepsilon, \varepsilon]^{n+hm}, \ \varepsilon=10^{-3}$ and add each perturbation to the initial condition.
Then, with each perturbed initial condition, we obtain trajectory $s^\mathrm{p}(t)$ with a step size of $\tau = 10^{-3}$.
We set $\delta s(t) = s^\mathrm{p}(t) - s(t)$ and calculate $V_\delta(t) \coloneqq \delta s (t)^\top M(s(t),t) \delta s (t)$ where $\kappa=1$ and $\gamma=0.1$.
Finally, we calculate $\dot{V}(t)$ using the centered difference and observe error $e(t) = d(t) - \dot{V}(t) - \gamma V(t)$ and error rate $r^e(t) = e(t)/V(t)$, which is expected to be $O(\|\delta s\|) + O(\|\tau\|)$.
The absolute-valued $r^e(t)$ in $0\leq t\leq 5$ among 100 perturbations are observed to be at most $3.93\times 10^{-3}$, which corresponds to the expected order.
This result verifies $\dot{V} + \gamma V = d$, which is the claim of Lemma \ref{thm:matrix_inequality}.

\section{Conclusion}

This letter has analyzed the contraction of the closed-loop system controlled by the continuation method.
Using our proposed metric, we have derived a pair of matrix inequalities that implies the contraction of the closed-loop system.
%One of the obtained matrix inequalities elucidates the impact of the suboptimality of the input on the contraction of the optimally controlled system.
We also presented a numerical example that verifies our analysis.

\bibliographystyle{IEEEtran}
\bibliography{references}


%%%%%%%%%%%%%%%%%%%%%%%%%%%%%%%%%%%%%%%%%%%%%%%%%%%%%%%%%%%%%%%%%%%%%%%%%%%%%%%%

%%%%%%%%%%%%%%%%%%%%%%%%%%%%%%%%%%%%%%%%%%%%%%%%%%%%%%%%%%%%%%%%%%%%%%%%%%%%%%%%

%%%%%%%%%%%%%%%%%%%%%%%%%%%%%%%%%%%%%%%%%%%%%%%%%%%%%%%%%%%%%%%%%%%%%%%%%%%%%%%%

\end{document}